\documentclass[12pt,twoside,a4paper]{article}
\usepackage[a4paper,left=25mm,right=25mm,bottom=25mm,top=25mm]{geometry}
\usepackage[T2A]{fontenc}
\usepackage[cp1251]{inputenc}
\usepackage[english, russian]{babel}
\usepackage{microtype}
\usepackage{amsmath}
\usepackage{amssymb}
\usepackage{amsfonts}
\usepackage{amsthm}
\usepackage{bm}
\usepackage{eucal}
\usepackage{centernot}
\usepackage{graphicx}
\usepackage{indentfirst}
\usepackage[integrals]{wasysym}
\usepackage{caption2}
\usepackage{mathptmx}
\usepackage{tempora}
\usepackage{trudyIM}
\usepackage[unicode]{hyperref}
\usepackage{xcolor}
\hypersetup{
colorlinks=true, 
linktoc=all,     
linkcolor=red,  
linkbordercolor=blue
}
\begin{document}

\setlength{\abovedisplayskip}{6pt} 
\setlength{\belowdisplayskip}{6pt} 
\setlength{\topsep}{0mm} 



\English 
\setcounter{section}{0}
\setcounter{equation}{0}
\setcounter{figure}{0}
\setcounter{footnote}{0}

\firstheader{}

\mainheaderodd{} 

\mainheaderodd{Factorization of central polynomials over division algebras} 

\mainheadereven{A.~G.~Goutor, S.~V.~Tikhonov} 

\pagestyle{main} 
\thispagestyle{first} 

\

\begin{center}
\large\bfseries FACTORIZATION OF CENTRAL POLYNOMIALS OVER DIVISION ALGEBRAS
\end{center}

{\let\thefootnote\relax\footnote{The research was conducted under the 2026 grant from the President of the Republic of Belarus in the field of science.}}

\medskip

\begin{center}
\large\bfseries A.~G.~Goutor$^1$, S.~V.~Tikhonov$^2$\label{ivanovstart}
\end{center}

\medskip

\begin{center}
\it $^1$Belarusian State University, Minsk, Belarus\\
\it $^2$Belarusian State University, Minsk, Belarus\\
e-mail: goutor7@gmail.com, tikhonovsv@bsu.by
\end{center}

\bigskip

\engabstract{The factorization problem is a classical topic in the theory of polynomial rings. In this paper, we obtain an irreducibility criterion for central polynomials over
division algebras. An example demonstrating the computation of factors for a central polynomial is also presented. }
\engkewords{polynomial ring over a division algebra, central polynomial, factorization of polynomials}

\noindent\begin{tabular}{@{}p{0.25\textwidth}p{0.725\textwidth}@{}}
\footnotesize\textbf{Keywords:} \insertengkeywords&
\footnotesize\textbf{Abstract.} \insertengabstract\\
\hline
\end{tabular}


\section{Introduction and notation}

Let $K$ be a field. Recall that a finite-dimensional associative unital $K$-algebra is called central simple if its center is $K$ and it has no non-trivial two-sided ideals.
Let $\cal{D}$ be a central division algebra with center $K$. Let $\cal{D}^{*}$ also denote its multiplicative group.
We denote by ${\cal{D}}[x]$ the ring of polynomials in $x$ with coefficients in $\cal{D}$. We assume that $x$ commutes with all elements of $\cal{D}$.
Thus, every polynomial in ${\cal{D}}[x]$ has the form

\begin{equation} \label{eq_p(x)}
a_nx^n+a_{n-1}x^{n-1}+\dots + a_1x+a_0,\quad a_0,\dots, a_n \in \cal{D}.
\end{equation}

Polynomial addition and multiplication in the ring ${\cal{D}}[x]$ are defined in the usual way.
The degree of a polynomial of the form (\ref{eq_p(x)}) is also defined conventionally and equals $n$ if $a_n\ne 0$.
The right division theorem with remainder holds in ${\cal{D}}[x]$, which allows us to define the greatest common right divisor 
for any polynomials $p(x),q(x) \in {\cal{D}}[x]$, and it can be found using the Euclidean algorithm (see \cite{Or33}).
It is known that $a\in {\cal{D}}$ is a right root of a polynomial $p(x) \in {\cal{D}}[x]$ if and only if
$x-a$ is a right divisor of $p(x)$ in ${\cal{D}}[x]$ (\cite[Proposition 16.2]{La91}), that is,  $p(x)=s(x)(x-a)$ for some polynomial $s(x)$ in ${\cal{D}}[x]$.

The basic properties of polynomials over division rings are described in \cite[Ch. 5, \S 16]{La91} (see also \cite{GoMo65}, \cite{BrWh83}).

The factorization problem is a classical topic in the theory of polynomial rings. Every monic polynomial in ${\cal{D}}[x]$ can be expressed as a product of irreducible polynomials, but
the theory in the non-commutative case is significantly more challenging and profoundly richer.
In particular, such polynomials do not possess a unique factorization into irreducible factors; however, any two factorizations of the same polynomial contain
the same number of irreducible factors (see, e.g., \cite[Chapter II, Theorem 1]{Or33}).

In the non-commutative setting, polynomials over the algebra of Hamiltonian quaternions $\mathbb {H}$ have been the most extensively studied (see \cite{GoMo65}, \cite{Ni41}).
By the Niven-Jacobson theorem \cite[Theorem 16.14]{La91}, the quaternion division algebra
over a real-closed field is algebraically closed, that is every polynomial in $\mathbb {H}[x]$ splits completely into a product of linear factors.
The problem of finding the roots of polynomials from $\mathbb {H}[x]$ was also considered, for example, in \cite{FaMiSeSo17}, \cite{HuSo02}, \cite{JaOp10}, \cite{SePeVi01}, \cite{SeSi01},
\cite{Ka13}, \cite{SaKoSo19}, \cite{GoTi25}.
Factorization algorithms for $\mathbb {H}[x]$ can be found in \cite{LiScSc19} and \cite{ScSc21}.
The case of generalized quaternion algebras has been recently considered in \cite{Ko23},
where the author presented algorithms for factoring polynomials over quaternion division algebras over number fields and for finding their roots.
The goal of this paper is to generalize some results from \cite{Ko23} to arbitrary central division algebras.

In particular, the following result was obtained in \cite{GoTi25}:

\begin{theorem} \label{th:factorization}
Every polynomial $p(x)\in {\cal{D}}[x]$ can be uniquely represented as
$$
p(x)=c g(x) h(x),
$$
where $c\in {\cal{D}}^*$ is the leading coefficient of $p(x)$, $h(x)$ is a monic polynomial with coefficients in the field $K$,
$g(x) \in {\cal{D}}[x]$ is a monic polynomial that has no non-constant right divisors in $K[x]$.
\end{theorem}

We call the polynomial $h(x)$ from the previous theorem the maximal central factor of $p(x)$.
In view of this theorem, factoring polynomials in ${\cal{D}}[x]$ reduces to factoring those with coefficients in $K$ that are irreducible in $K[x]$, and those with no non-trivial factors in $K[x]$.

In this paper, we consider the irreducibility and factorization of central polynomials (i.e., polynomials with coefficients in $K$) in ${\cal{D}}[x]$.
We obtain an irreducibility criterion for central polynomials over division algebras.
We also present an example demonstrating the computation of factors for a central polynomial over a symbol algebra of degree 3.

Throughout this paper, we adopt the following notation. For a central simple $K$-algebra ${\cal{A}}$, let $\mathrm{deg}(\mathcal{A})$ denote its degree (i.e., the square root of its dimension).
The reduced norm of an element $a\in {\cal{A}}$  is denoted by $\mathrm{Nrd}(a)$. In particular, for any $b\in K$, we have $\mathrm{Nrd}(b)=b^{\mathrm{deg}(\mathcal{A})}$.
For a polynomial $f(x) \in {\cal{D}}[x]$, $\mathrm{Nrd}(f(x)$) denotes the reduced norm of $f(x)$ in the central division algebra ${\cal{D}}\otimes_K K(x)$, where $K(x)$ is the field of rational functions in one variable.
Note that for any $f(x) \in {\cal{D}}[x]$, $\mathrm{deg} (\mathrm{Nrd}(f(x))) = \mathrm{deg}(\mathcal{D}) \mathrm{deg} (f(x))$.

\section{Irreducibility of central polynomials}

We begin with the following irreducibility criterion for central polynomials, which generalizes Proposition 6 from \cite{Ko23}.

\begin{theorem} \label{th:criterion}
Let ${\cal{D}}$ be a central division algebra with center $K$ and $f(x)\in K[x]$ be irreducible over $K$.
Let $L$ be the field $K[x]/(f(x))$.
Then $f(x)$ remains irreducible in ${\cal{D}}[x]$ if and only if the algebra ${\cal{D}}\otimes_K L$ has no zero divisors.
\end{theorem}

\begin{proof}
The homomorphism of $K$-algebras
$$
{\cal{D}}[x]\to {\cal{D}}\otimes_K L,
$$
$$
a_nx^n+\dots + a_1x+a_0 \mapsto a_n\otimes (x+(f(x)))^n+\dots + a_1\otimes (x+(f(x)))+a_0
$$
induces the isomorphism
$$
{\cal{D}}[x]/(f(x)) \cong {\cal{D}}\otimes_K L.
$$

Assume that ${\cal{D}}\otimes_K L$ has no zero divisors. If $f(x)=p(x)q(x)$ for some $p(x), q(x) \in {\cal{D}}[x]$, $\mathrm{deg}(p(x)), \mathrm{deg}(q(x)) > 0$,
then $p(x)+(f(x))$ and $q(x)+(f(x))$ are zero divisors in ${\cal{D}}[x]/(f(x))$. Contradiction.
	
Now assume that $f(x)$ remains irreducible in ${\cal{D}}[x]$.
If there are zero divisors in ${\cal{D}}[x]/(f(x))$, then there exist
$p(x), q(x)\in {\cal{D}}[x]$ such that $f(x)\nmid p(x)$, $f(x)\nmid q(x)$ and $p(x)q(x)=s(x)f(x)$ for some $s(x)\in {\cal{D}}[x]$.
Since $f(x)$ is an irreducible polynomial in ${\cal{D}}[x]$, the greatest common right divisor of $f(x)$ and $p(x)$ is 1.
By \cite[Theorem 4]{Or33}, there exist $A_1(x), A_2(x)\in {\cal{D}}[x]$ such that
$$
1=A_2(x)p(x)+A_1(x)f(x).
$$
Then
$$
q(x)=A_2(x)p(x)q(x)+A_1(x)f(x)q(x)=A_2(x)s(x)f(x)+A_1(x)q(x)f(x).
$$
Hence $f(x)\mid q(x)$. Contradiction.
\end{proof}



As a consequence, we obtain the following corollaries.

\begin{corollary}
Let $\cal{D}$ be a central division $K$-algebra of degree $n$.
Let $f(x)\in K[x]$ be an irreducible polynomial in $K[x]$ and $\mathrm{GCD}(\mathrm{deg} (f(x)), n)=1$. Then $f(x)$ remains irreducible in ${\cal{D}}[x]$.
\end{corollary}

\begin{proof}
Let $L$ be the field $K[x]/(f(x))$. Then $[L:K]= \mathrm{deg} (f(x))$.
By \cite[Proposition 13.4]{Pi82}, ${\cal{D}}\otimes_K L$ is also a division algebra. By Theorem \ref{th:criterion}, $f(x)$ is irreducible in ${\cal{D}}[x]$.
\end{proof}

\begin{corollary}
Let $\cal{D}$ be a central division $K$-algebra.  Let $f(x)\in K[x]$ be an irreducible polynomial in $K[x]$. Assume that $\mathrm{deg} ({\cal{D}})=\mathrm{deg} (f(x))$ is a prime number.
If $f(x)$ is reducible in ${\cal{D}}[x]$, then $f(x)$ is a product of linear factors in ${\cal{D}}[x]$.
\end{corollary}

\begin{proof}
Since $f(x)$ is reducible in ${\cal{D}}[x]$, by Theorem \ref{th:criterion}, the algebra ${\cal{D}}\otimes_K L$ has zero divisors, where $L=K[x]/(f(x))$.
Since the degree of this algebra is a prime number,
${\cal{D}}\otimes_K L$ is a matrix algebra over $L$. Hence $L$ is a splitting field of the algebra ${\cal{D}}$
and $L$ is isomorphic to a maximal subfield of the algebra ${\cal{D}}$ by \cite[Corollary 13.3]{Pi82}.
Since $f(x)$ has a root in $L$, $f(x)$ has a root in ${\cal{D}}$ as well. Therefore $f(x)$ is a product of linear factors in ${\cal{D}}[x]$ by \cite[Theorem 16.9]{La91}.
\end{proof}

Let ${\cal{U}}$ be a quaternion division algebra over a field $K$. Further, let
$f(x)\in K[x]$ be irreducible over $K$.
By \cite[Proposition 5]{Ko23}, either $f(x)$ remains irreducible in ${\cal{U}}[x]$ or 
$f(x)=\mathrm{Nrd}(p(x))$ for some irreducible quaternionic polynomial $p(x)\in {\cal{U}}[x]$.

In the case of algebras of higher degrees, this statement ceases to be true as the following example shows.

\begin{example}
Let ${\cal{D}}$ be a cyclic division $K$-algebra of degree 6 with maximal cyclic subfield $E$.
Let $F$ be a field such that $K \subset F \subset E$, $[F:K]=2$.

Let $a \in F, a \notin K$, and let $g(x)$ be the minimal polynomial of $a$ over $K$.
Then $\mathrm{deg}(g(x))=[F:K]=2$, $g(x)$ is irreducible over $K$, but $g(x)$ is reducible over $F$ (since $g(x)$ has a root $a \in F$).
Therefore $g(x)$ is reducible over $\cal{D}$. But $g(x)$ cannot be the norm of any polynomial in ${\cal{D}}[x]$, since its degree is less than 6.
\end{example}

However, in certain special cases, one can still obtain analogues of \cite[Proposition 5]{Ko23}.

\begin{proposition}
Let ${\cal{D}}$ be a central division $K$-algebra of degree 3. Further, let $f(x)\in K[x]$ be a monic irreducible in $K[x]$ polynomial.
Then either $f(x)$ remains irreducible in ${\cal{D}}[x]$ or $f(x)=\mathrm{Nrd}(s(x))$ for some irreducible polynomial $s(x)\in {\cal{D}}[x]$.
\end{proposition}

\begin{proof}
Assume that $f(x)=p(x)q(x)$ for some non-constant polynomials $p(x), q(x)\in {\cal{D}}[x]$. Then
$$
f(x)^3=\mathrm{Nrd}(f(x))=\mathrm{Nrd}(p(x))\mathrm{Nrd}(q(x)).
$$
Since $f(x)$ is irreducible in $K[x]$, either $f(x)=\mathrm{Nrd}(p(x))$ or $f(x)=\mathrm{Nrd}(q(x))$.
\end{proof}

\begin{proposition}
Let $\cal{D}$ be a central division $K$-algebra.  Let $f(x)\in K[x]$ be an irreducible polynomial in $K[x]$
and let $L$ be the field $K[x]/(f(x))$.
Assume that $\mathrm{deg} ({\cal{D}})=\mathrm{deg} (f(x))$
and ${\cal{D}}\otimes_K L$ is a matrix $L$-algebra.
Then $f(x)=\mathrm{Nrd}(p(x))$ for some linear polynomial $p(x)\in {\cal{D}}[x]$.
Moreover, $f(x)$ is a product of linear factors.
\end{proposition}

\begin{proof}
The field $L$ is a splitting field of the algebra ${\cal{D}}$
and $L$ is isomorphic to a maximal subfield of the algebra ${\cal{D}}$ by \cite[Corollary 13.3]{Pi82}.
Therefore $f(x)=\mathrm{Nrd}(x-a)$, where $a$ is a root of $f(x)$ in $L$.
\end{proof}

\section{Factorization of central polynomials}

In some cases, factors of a central polynomial can be found using ideas from Algorithm 2 in \cite{Ko23}.
Let ${\cal{D}}$ be a central division algebra with center $K$.

Suppose that $f(x)$ is an irreducible  polynomial  over $K$, let $L$ be the field $K[x]/(f(x))$.
If the algebra ${\cal{D}} \otimes_K L$ is division, then $f(x)$ remains irreducible in ${\cal{D}}[x]$ by Theorem \ref{th:criterion}.
Assume that ${\cal{D}} \otimes_K L \cong {\cal{D}}[x]/(f(x))$ has a zero-divisor $z$. Let $p(x)$ be the preimage  of $z$
under the homomorphism
$$
{\cal{D}}[x] \longrightarrow {\cal{D}}[x]/(f(x)).
$$
Then
$\mathrm{Nrd}(p(x))$ is divisible by $f(x)$.
Consequently, under certain conditions, computing the greatest common right divisor of the polynomials $f(x)$ and $p(x)$
allows one to successfully obtain a non-trivial divisor of $f(x)$.

Below, we provide an example illustrating how to compute factors of a central polynomial over a symbol algebra of degree 3.

\begin{example}

Let $K=\mathbb{Q}(\sqrt{-3})$ and let ${\cal{D}}=(2,7)_3$ be the symbol algebra of degree 3 over $K$.
That is, ${\cal{D}}$ is generated by two elements $X$ and $Y$ subject to the relations
$$
X^3 = 2, \quad Y^3 = 7, \quad YX = \zeta_3 XY,
$$
where $\zeta_3$ is a primitive cube root of unity.
The algebra $\cal{D}$ is a division algebra since one can show that its localization ${\cal{D}}\otimes_K \mathbb{Q}_7$ is a local division algebra, where $\mathbb{Q}_7$ denotes the field of 7-adic numbers.

The polynomial
$$
f(x) = x^{6}-72
$$
is irreducible in $K[x]$.
Indeed, by \cite[Chapter 6, Theorem 9.1]{La02}, it suffices to show that $72$ is neither a square nor a cube in the field $K$.
It is easy to see that $72$ is not a square in $K$ since $\sqrt{2} \notin K$.
Furthermore, since $72$ differs from $9$ by a perfect cube in $\mathbb{Q}$, it is a cube in $K$ if and only if $9$ is.
This is impossible because $[\mathbb{Q}(\sqrt[3]{9}) : \mathbb{Q}] = 3$, which does not divide $[K : \mathbb{Q}] = 2$.

Let $L = K(a)$, where $a$ is a root of the polynomial $f(x)$. Then $L \cong K[x]/(f(x))$ and,
using the Albert--Brauer--Hasse--Noether theorem (see, e.g., \cite[\S~18.4]{Pi82}), one can show that
${\cal{D}} \otimes_K L$ is a split algebra. Hence, by Theorem \ref{th:criterion}, the polynomial $f(x)$ is reducible in ${\cal{D}}[x]$.

There is an explicit isomorphism of $L$-algebras
$$
\mathcal{D}\otimes _{K}L \longrightarrow M_{3}(L)
$$
sending the generators $X$ and $Y$ to the following matrices:
$$
X \mapsto \begin{pmatrix}\frac{a^{2}}{2}&0&-7\zeta _{3}\\ -1&\frac{a^{2}\zeta _{3}^{2}}{2}&0\\ 0&-\zeta _{3}^{2}&\frac{a^{2}\zeta _{3}}{2}\end{pmatrix}, \quad
Y \mapsto \begin{pmatrix} 0 & 0 & 7 \\ 1 & 0 & 0 \\ 0 & 1 & 0 \end{pmatrix}.
$$

The preimage of the matrix
$$
\begin{pmatrix} 1 & 0 & 0 \\ 0 & 0 & 0 \\ 0 & 0 & 0 \end{pmatrix}
$$
under this isomorphism is given by the element
$$
\frac{1}{3} + \frac{a^4}{24} X + \frac{a^2 \zeta_3}{12} XY + \frac{\zeta_3^2}{6} XY^2 + \frac{a^2}{12} X^2 - \frac{1}{6} X^2Y.
$$

Thus, the element
$$
8 + a^4 X + 2a^2 \zeta_3 XY + 4\zeta_3^2 XY^2 + 2a^2 X^2 - 4X^2Y
$$
is a zero divisor in ${\cal{D}} \otimes_K L$.

The preimage of this zero divisor under the homomorphism ${\cal{D}}[x] \longrightarrow {\cal{D}} \otimes_K L \cong {\cal{D}}[x]/(f(x))$
is the polynomial
$$
p_1(x):= X \cdot x^4 + \left( 2X^2 + 2\zeta_3 XY \right) x^2 + \left( 8 + 4\zeta_3^2 XY^2 - 4X^2Y \right).
$$

One can check that
$$
\mathrm{Nrd}(p_1(x)) = 2(x^{6}-72)^{2} = 2 f(x)^{2}.
$$

Consequently, $p_1(x)$ is a factor of $f(x)^2$ in ${\cal{D}}[x]$.
A routine calculation shows that $f(x)$ is divisible by $p_1(x)$ both from the left and from the right, yielding the  factorization
$$
x^{6}-72 = \left( X x^4 + \left( 2X^2 + 2\zeta_3 XY \right) x^2 + \left( 8 + 4\zeta_3^2 XY^2 - 4X^2Y \right) \right) \cdot \left( \frac{1}{2}X^2 x^2 - X^2Y - 2 \right) =
$$
$$
\left( \frac{1}{2}X^2 x^2 - X^2Y - 2 \right) \cdot \left( X  x^4 + \left( 2X^2 + 2\zeta_3 XY \right) x^2 + \left( 8 + 4\zeta_3^2 XY^2 - 4X^2Y \right) \right)
$$
in ${\cal{D}}[x]$.

Similarly, taking the preimage of the matrix unit
$$
\begin{pmatrix} 0 & 0 & 0 \\ 0 & 1 & 0 \\ 0 & 0 & 0 \end{pmatrix},
$$
which is given by the element
$$
\frac{1}{3} + \frac{a^4 \zeta_3}{24} X + \frac{a^2 \zeta_3^2}{12} XY + \frac{1}{6} XY^2 + \frac{a^2 \zeta_3^2}{12} X^2 - \frac{\zeta_3^2}{6} X^2Y
$$
we obtain the polynomial
$$
p_2(x) :=
 \zeta_3 X x^4 + \left( 2\zeta_3^2 X^2 + 2\zeta_3^2 XY \right) x^2 + \left( 8 + 4 XY^2 - 4 \zeta_3^2 X^2Y \right).
$$
This leads to another factorization of the polynomial $f(x)$:
$$
x^{6}-72 =
$$
$$
\left(  \zeta_3 X  x^4 + \left( 2\zeta_3^2 X^2 + 2\zeta_3^2 XY \right) x^2 + \left( 8 + 4 XY^2 - 4 \zeta_3^2 X^2Y \right) \right) \cdot \left( \frac{\zeta_3^2}{2}X^2 x^2 - \zeta_3^2 X^2Y - 2
 \right)
=
$$
$$
\left( \frac{\zeta_3^2}{2}X^2 x^2 - \zeta_3^2 X^2Y - 2
 \right) \cdot \left( \zeta_3 X x^4 + \left( 2\zeta_3^2 X^2 + 2\zeta_3^2 XY \right) x^2 + \left( 8 + 4 XY^2 - 4 \zeta_3^2 X^2Y \right) \right).
$$
\end{example}

\vspace{-2mm}

\renewcommand{\refname}


{\normalsize\bf References}
\begin{thebibliography}{00}
\leftskip=-7mm
\parskip=-0mm
\parsep=0mm
\itemsep=0mm
\labelwidth=-12mm


\bibitem{Or33} Ore O. Theory of non-commutative polynomials.  {\it Ann. of Math. (2)}, 1933, vol.\,34, no.~3, pp.\,480–-508.

\bibitem{La91}
Lam T. Y. {\it A first course in noncommutative rings.} Graduate Texts in Mathematics 131. Springer-Verlag, New York, 1991.

\bibitem{GoMo65} Gordon B., Motzkin T.S. On the zeros of polynomials over division rings.  {\it Trans. Amer. Math. Soc.}, 1965, vol.\,116, pp.\,218--226.

\bibitem{BrWh83}
Bray U., Whaples G. Polynomials with coefficients from a division ring.  {\it Can. J. Math.}, 1983, vol.\,35, pp.\,509--515.

\bibitem{Ni41}
Niven I. Equations in quaternions. {\it Amer. Math. Monthly}, 1941,  vol.\,48, no.~10, pp.\,654–661.


\bibitem{FaMiSeSo17}
Falc\~{a}o M. I., Miranda F., Severino R., Soares M. J.  Mathematica Tools for Quaternionic Polynomials.
{\it Computational science and its applications.}  ICCSA 2017. Part II. Lecture Notes in Comput. Sci., 10405 Springer, Cham, 2017, pp.\,394--408.

\bibitem{HuSo02}
Huang L., So W. Quadratic formulas for quaternions.  {\it Appl. Math. Lett.}, 2002, vol.\,15, no.~5, pp.\,533--540.

\bibitem{JaOp10} Janovsk\'{a} D., Opfer G. A note on the computation of all zeros of simple quaternionic polynomials.
{\it SIAM J. Numer. Anal.}, 2010,  vol.\,48, no.~1, pp.\,244--256.

\bibitem{SePeVi01} Ser\^{o}dio R., Pereira E., Vit\'{o}ria J. Computing the zeros of
quaternion polynomials.  {\it Comput. Math. Appl.}, 2001,  vol.\,42, no.~8-9,  pp.\,1229--1237.

\bibitem{SeSi01} Ser\^{o}dio R., Siu L.-S. Zeros of quaternion polynomials.
{\it Appl. Math. Lett.}, 2001, vol.\,14,  no.~2, pp.\,237--239.

\bibitem{Ka13} Kalantari B., Algorithms for quaternion polynomial root-finding. {\it J. Complexity}, 2013, vol.\,29,  no.~3-4, pp.\,302–-322.

\bibitem{SaKoSo19} Sakkalis T., Ko K., Song G. Roots of quaternion polynomials: theory
and computation. {\it Theoret. Comput. Sci.}, 2019, vol.\,800, pp.\, 173--178.

\bibitem{GoTi25} Goutor A. G., Tikhonov S. V. Polynomials over division rings. (Russian) {\it Tr. Inst. Mat.}, 2025, vol.\,33, no.~2, pp.\, 13--20.

\bibitem{LiScSc19} Li Z., Scharler D.F., Schr\"{o}cker H.-P.  Factorization results for left polynomials in some associative real algebras: state of the art, applications, and open questions.
{\it J. Comput. Appl. Math.}, 2019, vol.\,349, pp.\, 508–522.

\bibitem{ScSc21}
Scharler D.F., Schr\"{o}cker H.-P.  An algorithm for the factorization of split
quaternion polynomials. {\it Adv. Appl. Clifford Algebr.}, 2021,   vol.\,31, no.~3, Paper No. 29.

\bibitem{Ko23} Koprowski P. Factorization and root-finding for polynomials over division quaternion algebras.
{\it  Proceedings of the International Symposium on Symbolic \& Algebraic Computation} (ISSAC 2023), 417--424, ACM, New York, [2023].

\bibitem{Pi82} Pierce R. S. {\it Associative algebras.}  Studies in the History of Modern Science, 9. Graduate Texts in Mathematics, 88. Springer-Verlag, New York-Berlin, 1982. 

\bibitem{La02} Lang S. {\it Algebra.} Graduate Texts in Mathematics 211. Springer-Verlag, New York, 2002.

\label{ivanovend}
\end{thebibliography}
\end{document}